\documentclass{amsart}
\usepackage[foot]{amsaddr}
\usepackage[utf8]{inputenc}
\usepackage{a4wide}
\usepackage{amsmath,amsxtra,amssymb,latexsym, amscd,amsthm,nicefrac,mathrsfs,bbm,mathtools, tikz-cd}
\usetikzlibrary{cd}
\usepackage[symbol]{footmisc} 
\usepackage{dutchcal}

\newtheorem{theorem}{Theorem}[section]
\newtheorem{corollary}[theorem]{Corollary}
\newtheorem{lemma}[theorem]{Lemma}
\newtheorem{proposition}[theorem]{Proposition}
\newtheorem{definition}[theorem]{Definition}
\newtheorem{remark}[theorem]{Remark}

\numberwithin{equation}{section}

\newcommand{\onehalf}{\nicefrac{1}{2}}
\newcommand{\norm}[1] {\| #1 \|}

\newcommand{\bignorm}[1]{\bigl\| #1 \bigr\|}
\newcommand{\Bignorm}[1]{\Bigl\| #1 \Bigr\|}

\newcommand{\CC} {\mathbb C}
\newcommand{\RR} {\mathbb R}

\newcommand{\ud} {{\mathrm{d}}}

\newcommand{\dr} {{\,\mathrm{d}r}}
\newcommand{\ds} {{\,\mathrm{d}s}}
\newcommand{\dt} {{\,\mathrm{d}t}}

\newcommand{\DOMAIN}{\mathscr D} 
\newcommand{\BOUNDED}{\mathscr L}

\renewcommand{\Re}{\text{Re}}
\newcommand{\Ran}{\text{range}}

\newcounter{aufzi}
\newenvironment{aufzi}{\begin{list}{ {\upshape\alph{aufzi})}}{
        \usecounter{aufzi}
        \topsep0ex
        \parsep0cm
        \itemsep0ex
        \leftmargin1cm
        \labelwidth0.5cm
        \labelsep0.3cm
        \itemindent0cm
}}
{\end{list}}

\begin{document}

\allowdisplaybreaks

\title[Averaged Hautus test for non-autonomous evolution equations]{Controllability and observability for non-autonomous evolution equations: the averaged Hautus test}

\date{\today} %

\author{Bernhard Haak} %
\thanks{The first named author is partially supported by the
   Collaborative Research DFG and ANR project INFIDHEM  ANR-16-CE92-0028  }  
\email{bernhard.haak@math.u-bordeaux.fr}

\author{Duc-Trung Hoang} %
\email{duc-trung.hoang@math.u-bordeaux.fr}

\author{El Maati Ouhabaz} %
\email{elmaati.ouhabaz@math.u-bordeaux.fr}

\address{Institut de Math\'ematiques de Bordeaux, UMR CNRS 5251, Universit\'e de  Bordeaux,  351 cours de la Liberation,  33405 Talence,  {\sc France}}

\begin{abstract}
  We consider the observability problem for non-autonomous evolution
  systems (i.e., the operators governing the system depend on
  time). We introduce an averaged Hautus condition and prove that for
  skew-adjoint operators it characterizes exact observability. Next,
  we extend this to more general class of operators under a growth
  condition on the associated evolution family. We give an application
  to the Schrödinger equation with time dependent potential and the
  damped wave equation with a time dependent damping coefficient.
\end{abstract}

\maketitle

\section{Introduction}\label{sec:intro}

Observability is an important  concept in system and control theory. It
treats the question to which extent an observation, i.e.,  partial
knowledge of the solution of an evolution equation, determines its
initial or final state. The theory has been studied for several decades for systems of the form:
\begin{equation}\label{eq:observed-autonomous-evolution-equation}
 \left\{
   \begin{array}{lcl}
   x'(t) + A x(t)&=& 0 \quad t \in [0,T] \\ x(0) &=& x_0 \\  y(t) &=& C x(t)
   \end{array}\right.
 \end{equation}
in which the two operators $A$ and $C$ are independent of time $t$ and satisfy appropriate conditions such as 
$-A$, with domain  $\DOMAIN(A)$, generates a strongly continuous semigroup on a Hilbert space $H$ and $C$ is bounded from $\DOMAIN(A)$ into another Hilbert space $Y$. 

Observability  consists of unique determination or recovery of the initial (or final) time state
under the  knowledge of the observed solution $y(\cdot)$. Recall
that in the case of matrices $A$ and $C$ (finite dimensional setting), all observation
concepts coincide and can be characterized in various manners. The
Kalman rank condition is certainly the most known version; it states
that $C$ is observable if and only if the matrix
\[
  [ C \,|\, CA \,|\, CA^2 \,|\, \ldots \,|\, CA^{n-1} ]
\]
has full rank. An equivalent statement is the Hautus lemma: it characterizes
observability by the condition 
\[
\forall \lambda \in \CC: \qquad   {\displaystyle \operatorname {rank} [\lambda I -A,C ]=n}
\]
that clearly is equivalent to the condition
\begin{equation}  \label{eq:Hautus-finite-dim}
 \norm{ C x}^2 +  \norm{ (\lambda I - {A}) x }^2  \ge \kappa \norm{x}^2.
\end{equation}
In an infinite-dimensional setting with operators $A, C$, instead of
matrices, rank conditions are not appropriate. However, the Hautus test in
the form (\ref{eq:Hautus-finite-dim}) can be generalized, and has
actually been proposed in \cite{russell} as a criterion for
observability. %
Russell and Weiss conjectured in \cite{russell} that this inequality
characterizes exact observability. They proved in \cite{russell} that
the conjecture is valid for bounded and invertible operators
$A$. Later, Jacob and Zwart \cite{Jacob2} showed equivalence for
diagonal semigroup generator on a Riesz basis if the output space $Y$
is finite dimensional. The general conjecture was later proved to be
wrong, see \cite{JZ:counterexample}. Note however, that if
$C$ is admissible and $A$ has a bounded $H^\infty$-calculus on a
suitable sector (which is, in turn a consequence of admissibility and
exact observation, see Proposition 5.1 in \cite{HaakOuhabaz:JFA}),
then it does not seem to be known whether the Hautus condition implies
observability. There exist other formulations of the Hautus condition
(or spectral condition) and there are several cases where it implies
exact observability. This holds for example if $A$ generates a unitary
group. We refer to \cite{Liu,ZhouYamamoto} for early results with
bounded observations, and \cite{BurqZworski, Miller} for successive
extensions. These have subsequently been generalized (see
\cite{Jacob3}) to groups with certain growth bounds. See also
\cite{Tucsnak-Weiss:book} for more information and references on this
subject.

In this paper we consider first order
non-autonomous evolution equations of the following form:
  \begin{equation}\label{eq:observed-non-autonomous-evolution-equation}\tag{A,C}
 \left\{
  \begin{array}{lcl}
  x'(t) + A(t) x(t)&=& 0 \quad t \in [0,\tau] \\ x(0) &=& x_0 \\  y(t) &=& C(t)x(t).
   \end{array}\right.
 \end{equation}
 The difference with \eqref{eq:observed-autonomous-evolution-equation}
 is that we allow operators $A$ and $C$ to depend on time $t$. To be
 precise, let $\tau >0$ and assume that for $t \in [0,\tau]$, the
 operator $A(t)$ generates a strongly continuous contraction semigroup
 $(e^{-sA(t)})_{s\ge 0}$ on the Hilbert space $H$. We suppose further
 that there exists a densely and continuously embedded subspace
 $\DOMAIN \hookrightarrow H$ such that that for all $t \in [0,T]$,
 $\DOMAIN(A(t)) = \DOMAIN$ and that $t \mapsto A(t)v$ is continuously
 differentiable in $H$ for every $v \in \DOMAIN$.  These assumptions
 are sufficient to guarantee that the Cauchy problem
 $x'(t) = A(t) x(t)$, $x(0)=x_0$ admits a solution, see e.g.
 \cite[Sections 5.3 and 5.4]{Pazy:book}.  For each $t$,
 $C(t): \DOMAIN \rightarrow Y$ is a bounded operator. Then, for
 initial data $x_0 \in \DOMAIN$, the solution $x$ to
 (\ref{eq:observed-non-autonomous-evolution-equation}) satisfies
 $x(t) \in \DOMAIN$ for each $t \ge 0$ and hence $y(t)$ is well
 defined. We define observability concepts (and controllability
 concepts for the adjoint system) as in the autonomous case
 \eqref{eq:observed-autonomous-evolution-equation}.  

 In the case of time-dependent matrices, a famous result of Silverman
 and Meadows \cite{silverman} characterizes exact observability and
 controllability. Their arguments have been adapted to certain
 infinite dimensional settings, see for example \cite{ACO,ABDGB,AWD}.
 Our main objective is different. We seek to prove observability from a
 certain Hautus type condition. In order to do this, we introduce the
 following averaged Hautus conditions:
\begin{equation*}
   \norm{x}^2 \leq m^2 \Bigl( \tfrac{1}{\tau} \int_{0}^{\tau}\bignorm{C(s)e^{\lambda s}x}^2 \ds\Bigr) + M^2\Bigl( \tfrac{1}{\tau}\int_{0}^{\tau}e^{\Re \lambda . s}\bignorm{(\lambda{+}A(s))x}\ds\Bigr)^2
\end{equation*}
for all $\lambda \in \mathbb{C}$ and all $x \in \DOMAIN$,  or  
\begin{equation*} 
\norm{x}^2 \leq m^2  \Bigl( \tfrac{1}{\tau} \int_{0}^{\tau}\bignorm{C(s)x}^2\ds\Bigr) + M^2 \Bigl(\tfrac{1}{\tau}\int_{0}^{\tau}\bignorm{(i\xi + A(s))x}^2 \ds\Bigr)
\end{equation*}
for all $\xi \in \mathbb{R}$ and $x \in \DOMAIN$. 
These inequalities do coincide with the usual Hautus conditions if the
operators $A$ and $C$ are independent of $t$.  We prove that these
averaged Hautus conditions imply exact observability when the
operators $A(t)$ are skew-adjoint. This result is refined to the case
of invertible evolution families (not necessarily unitary) under
certain growth constraints. We apply these results to Schr\"odinger
equations with time dependent potentials and to a damped wave-equation
with time-dependent damping. \\
Finally, we mention the papers \cite{Emanuilov}, \cite{ABDGB}and the references therein  on observability (or controllability) of parabolic equations (with time dependent coefficients). The approach in these papers is based on Carleman estimates and it differs from ours.

\section{Preliminary results}\label{sec:prelim-reults}

Recall that we suppose $A(t) : \DOMAIN \to H$ to have a fixed domain,
that $t \mapsto A(t)v$ is continuously differentiable in $H$ for every
$v \in \DOMAIN$ and each semigroup $e^{-sA(t)}$ is a contraction on
$H$. By \cite[Sections 5.3 and 5.4]{Pazy:book} there exists a unique
evolution family $(U(t,s))_{0\leq s \leq t \leq \tau }$ on $H$
generated by $A(t)_{0\le t \le \tau}$.  This evolution family satisfies
the following properties.
 \begin{enumerate}
 \item $\norm{U(t,s)} \leq Me^{-\omega(t-s)}$ for some $\omega \in
   \mathbb{R}$
 \item For all $v \in \DOMAIN$, $\frac{\partial^{+}}{\partial t}U(t,s)v|_{t=s} = -A(s)v,  \quad 
    \frac{\partial^{+}}{\partial t}U(t,s)v = -A(t)U(t,s)v$.
\item For all $v \in \DOMAIN$, $\frac{\partial}{\partial s}U(t,s)v = U(t,s)A(s)v$. 
\item $U(t,s)\DOMAIN \subseteq \DOMAIN$
\item For all $ v \in \DOMAIN$, $(s, t) \mapsto  U(t,s)v$ is continuous in $\DOMAIN$ for
  $0 \leq s \leq t \leq T$.
 \end{enumerate}
 For every $ v \in \DOMAIN$, the evolution equation
\begin{equation}\label{eq:evolution-equation}\tag{CP}
 \left\{
   \begin{array}{lcl}
   \frac{d}{\dt}\eta(t) + A(t) \eta(t)&=& 0 \quad 0 \leq s \leq t \leq \tau  \\
   \eta(s) &=& v 
   \end{array}\right.
 \end{equation}
 has a unique solution. This solution is given by $\eta(t) = U(t,s)v$.
 For $f \in L^1(0,\tau ;H)$,  the non homogeneous problem 
 \begin{equation}\label{eq:non-homogeneous-problem}\tag{NHCP}
 \left\{
   \begin{array}{lcl}
   \frac{d}{\dt}\eta(t) + A(t) \eta(t)&=& f(t) \quad 0 \leq s \leq t \leq \tau  \\
   \eta(s) &=& v \in H.
   \end{array}\right.
\end{equation}
 has then a  mild solution given by 
 \begin{equation} \label{eq:mild-solution}
 \eta(t) = U(t,s)v + \int_{s}^{t} U(t,r)f(r) \, \dr,
\end{equation}
see e.g.  \cite[p.146]{Pazy:book}. If, in addition to the standing
assumptions, $f \in C^1([s,\tau],H)$ then
\eqref{eq:non-homogeneous-problem} has a unique classical solution
which coincides with the mild solution, see for example
\cite[Theorem~5.2, p.146]{Pazy:book}.

We associate with \eqref{eq:observed-non-autonomous-evolution-equation} the operator
\[
 (\Psi_{s,\tau} x)(t)   =   \;    \left\{ \begin{array}{ll}
   C(t)U(t, s) x \qquad &t \in [s,\tau]  \\ 0 \quad   &t > \tau 
   \end{array}\right. 
\]
and define the following notions:

\begin{definition}[{Averaged admissible observations}] 
  Let $(C(t))_{t\in [0,\tau]}$ be a family of bounded operators in
  $\BOUNDED(\DOMAIN, Y)$, where $Y$ is some Hilbert space. We say that
  $(C(t))_{t}$ are averaged admissible observations for
  $(A(t))_{t \in [0,\tau]}$ if there exists a constant $ M_{\tau} >0$
  such that 
\[
 \int_{s}^{\tau} \bignorm{C(t) U(t,s)x}^2 \dt \leq M_{\tau}^2 \norm{x}^2 \quad \forall x \in \DOMAIN, s \in [0, \tau].
\]
(one can also consider a weaker admissibility notion by requiring the
above inequality for $s= 0$, only). For a single operator $C(t_0)$
such that
\[
 \int_{0}^{\tau} \bignorm{C(t_0)U(t,s)x}^2 \dt \leq M_{\tau}\norm{x}^2 \quad \forall x \in \DOMAIN
\]
we say that $C(t_0)$ is admissible for $(A(t))_{t\in [0, T]}$. 
\end{definition}  

For averaged admissible observations, $\Psi_{s,\tau}$ extends to a bounded
operator from $H$ to $L_2(s, \tau; Y)$ which we  denote again  by $\Psi_{s,\tau}$.

In this definition the norm inside the integral is taken in $Y$ and
the norm of $x$ is taken in $H$. We always use the same notation
$\norm{\cdot}$ for both, the difference will be clear from the
context.

\begin{definition} Suppose that $(C(t))_t$ is an averaged admissible observation for  $(A(t))_t$. We say that  the system $(A,C)$ is
  \begin{aufzi}
  \item  {\em  exactly averaged observable in time $\tau$} if the map $\Psi_{s,\tau}$ is
    bounded from below in the sense that there exists a constant
    $ \kappa_{\tau} > 0$ such that for all $x \in \DOMAIN$ 
 \[
         \int_{0}^{\tau} \bignorm{C(t)U(t,0)x}^2 \dt \geq \kappa_{\tau}\norm{x}^2. 
 \]
 For a given $t_0 \in [0,\tau]$, the system $(A,C(t_0))$ is exactly observable  at time $\tau$ if 
 \[
\int_{0}^{\tau} \bignorm{C(t_0)U(t,0)x}^2 \dt \geq \kappa_{\tau}\norm{x}^2.
 \]
\item {\em  final-time averaged observable in time $\tau$ } if there exists a constant $\kappa_\tau  >0$ such that 
 \[
 \int_{0}^{\tau} \bignorm{C(t)U(t,0)x}^2 \dt \geq \kappa_{\tau}\norm{U(\tau,0)x}^2 \quad \forall x \in \DOMAIN.
 \]
 As above we define final observability for the simple operator $C(t_0)$ for some $t_0$ as
 \[
  \int_{0}^{\tau} \bignorm{C(t_0)U(t,0)x}^2 \dt \geq \kappa_{\tau}\norm{U(\tau,0)x}^2. 
 \]
\item {\em  approximately averaged-observable in time $\tau$ } 
  if $\ker \Psi_{s,\tau} = \{0\}$ for all $0 \le s < \tau$.
  Again we  define approximate observability for a single operator
  $C(t_0)$ if $(A,C(t_0))$ is approximate observable in average as
  above.
\end{aufzi}
\end{definition}

In order to justify the use of the term "averaged" in the previous notions of
observability, we note that it might be possible that $(A,C(t_0))$ is
not exactly (or final or approximately) observable for some $C(t_0)$
or even for all $ t_0 \in J $ for some subset $J$ of $ [0,\tau]$ but $(A,C)$ is exactly (or
final or approximately) observable in average. In order to see this, we
consider the autonomous case $A(t) = A$ and an observation
operator $C$ such that the autonomous system is exactly (or null or
approximately) observable at time $\tau_0$. Define
\begin{equation*}
C(t) = \left\{
   \begin{array}{lcl}
   C, \quad t \in [0,\tau_0] \\ 0, \quad t \in (\tau_0,\tau].
   \end{array}\right.
 \end{equation*}
 Then
 \[ 
 \int_{0}^{\tau} \bignorm{C(t)e^{-tA} x}^2 \dt  \geq  \int_{0}^{\tau_0} \bignorm{C(t) e^{-tA} x}^2 \dt  \geq \kappa_{\tau}\norm{x}^2.
 \]
 Hence  the averaged observability property for $(A, C(t))$
 at time $\tau$ holds but the system $(A, C(t_0))$ is not observable for
 $t_0 \in (\tau_0, \tau]$ at any time. The same observation is valid  for null and approximate average observability.

\bigskip

Along with (\ref{eq:observed-non-autonomous-evolution-equation}) we
consider a controlled evolution equation. First, we recall the following: one can
construct an extrapolation space $H_{-1}$ and extrapolated operators
$A_{-1}(t)$ such that the following diagram commutes
\begin{center}
\begin{tikzcd}
H \arrow[rrr, "A_{-1}(t)"] & &  & H_{-1}(t)  \\
\DOMAIN \arrow[rrr, "A(t)"]  \arrow[u, hook, "i"] & & & H \arrow[u, hook, "i"]
\end{tikzcd}
\end{center}
One way to realize $H_{-1}(t)$ is to take the  completion of $H$ with respect to a resolvent norm
$\norm{ (\lambda-A(t))^{-1} x }_H$ or via its identification with $\DOMAIN( A(t)^* )'$. For all this we refer to \cite[Chapter II.5]{EngelNagel}.

\medskip

In order to keep the abstract setting simple we will suppose for the rest of
this section that $\DOMAIN(A(t)^*) =: \DOMAIN^*$ is independent of
time as well and equivalent norms with constants independent of $t$. Note that if for all $t \in [0, \tau]$, $A(t) = A(0) + R_t$ with a bounded operator  on $H$, then 
$A(t)^* = A(0)^* + R_t^*$ with domain $\DOMAIN^* := \DOMAIN(A(0)^*)$ independent of $t$. In the setting of the averaged Hautus test we consider later, we will make the assumption $A(t) = A(0) + R_t$ with a family of uniformly bounded operators $R_t$ on $H$.   
In this case $H_{-1}(t) = H_{-1}$  and have equivalent norms with constants independent of  $t$. 

Let $U$ be another Hilbert space and let  $B(t): U \rightarrow H_{-1}$ is bounded for each $t \in
 [0,\tau]$. We consider  in $H_{-1}$  the evolution equation
 \begin{equation}\label{eq:controlled-non-autonomous-evolution-equation}\tag{A,B}
 \left\{
   \begin{array}{lcl}
   x'(t) + A(t) x(t)&=& B(t)u(t) \qquad t \in [0,\tau] \\ x(s) &=& 0. 
   \end{array}\right.
 \end{equation}
 Since the mild solution is of the form
 (\ref{eq:mild-solution}), we have the naturally associated operator
 \begin{equation}
    \Phi_{s, \tau}u = \int_{s}^{\tau}   U(\tau,r) B(r) u(r) \dr \,  \qquad  (\tau \le \tau)
  \end{equation}
  to (\ref{eq:controlled-non-autonomous-evolution-equation}).

\begin{definition}[{Averaged admissible controls}] 
  Let $(B(t))_{t\in [0,\tau]}$ be a family of bounded operators in
  $\BOUNDED(U; H_{-1})$.  We say that
  $(B(t))_{t}$ are averaged admissible controls for
  $(A(t))_{t \in [0,\tau]}$ if there exists a constant $ M_{\tau} >0$
  such that the solution $x$ to
  (\ref{eq:controlled-non-autonomous-evolution-equation}) satisfies
  $x(t) \in H$ and for all $s \in [0, \tau)$  
\[
 \Bignorm{ \int_{s}^{\tau} U(\tau,r) B(r) u(r) \dr }^2 \leq M_{\tau}^2   \bignorm{u}_{L_2(s, \tau; U)}^2 
\]
for all $u \in {\mathscr D}(0, \tau; U)$ (one can also consider a
weaker admissibility notions by requiring the above inequality for
$s=0$, only).
\end{definition}

Let us consider the retrograde final-value problem
\begin{equation}   \label{eq:retrograde}
\left\{
  \begin{array}{lcl}
  z'(t) {\bf -} A(t)^* z(t) &=& 0 \\
  z(\tau) &=& z_\tau.
  \end{array}\right.
\end{equation}
Observe that for $x\in \DOMAIN$ and $x^* \in \DOMAIN^*$, 
\[
   \tfrac{\ud}{\dt} \langle x,   U(\tau, t)^* x^* \rangle
 = \tfrac{\ud}{\dt} \langle U(\tau, t) x,   x^* \rangle
 = - \langle U(\tau, t) A(t) x, x^*  \rangle
=  \langle x,  {-}A(t)^* U(\tau, t)^* x^* \rangle
\]
so that $z(t) = U(\tau, t)^* z_\tau $ solves the retrograde equation
(\ref{eq:retrograde}) on $[s, \tau]$ for all $0 \le s < \tau$.

\begin{lemma}\label{lem:admiss-equivalence}
  The family $(B(t))_{t\in [0, \tau]}$ are admissible controls for
  $(A(t))_{t\in [0, \tau]}$ if and only if the family
  $(B(t)^*)_{t\in [0, \tau]}$ are admissible observations for the
  retrograde equation (\ref{eq:retrograde}).
\end{lemma}
\begin{proof}
The following calculation is standard.
\begin{align*}
         \sup_{\norm{u}_2 \le 1} \Bignorm{ \int_{s}^{\tau} U(\tau,r) B(r) u(r) \dr }
  = & \; \sup_{\norm{u}_2 \le 1} \sup_{\norm{x^*}\le 1 } \Bigl| \int_{s}^{\tau}  \langle U(\tau,r) B(r) u(r),  x^* \rangle \dr \Bigr|\\
  = & \; \sup_{\norm{x^*}\le 1 } \sup_{\norm{u}_2 \le 1} \Bigl| \int_{s}^{\tau}  \langle u(r), B(r)^* U(\tau, r)^* x^* \rangle \dr \Bigr|\\
  = & \; \sup_{\norm{x^*}\le 1 } \Bigl( \int_s^\tau \bignorm{  B(r)^* U(\tau, r)^* x^* }^2 \dr \Bigr)^{\onehalf}. \qedhere
\end{align*}
\end{proof}

\begin{definition} Let $(B(t))_{t}$ be averaged admissible controls for
  $(A(t))_{t \in [0,\tau]}$. We say that (\ref{eq:controlled-non-autonomous-evolution-equation}) is
  \begin{aufzi}
  \item {\em Exactly averaged controllable in time $\tau$} if for any 
    $s \in [0, \tau)$ and $x_s, x_\tau \in H$, there exist
    $u \in L^2(s,\tau;U)$ such that the mild solution $x$ satisfies
    $x(s) = x_s$ and $x(\tau) = x_\tau$.\\
     This definition coincides with the usual one in the autonomous case, that is, given two states 
     $x_s, x_\tau \in H$ we find a control $u$ such that the solution  takes the value  $x_s$ at the initial time
     $t= s$ and the value $x_\tau$ at time $t = \tau$. 
\item {\em approximately averaged controllable in  time $\tau$ } if 
  for any $0 \le s < \tau$ and any $x_s, x_\tau \in H$ and $\varepsilon > 0$, there exist
  $u \in L^2(0,\tau;U)$ such that $x(s) = x_s$ and
  $\norm{x(\tau) - x_\tau} < \varepsilon$.  
\item {\em averaged null controllable in time $\tau$ } if for every
  $ 0 \le s < \tau$ and every $x_s \in H$, there exist
  $u \in L^2(s,\tau;U)$ such that the mild solution $x$ satisfies
  $x(s) = x_s$ and $x(\tau) = 0$.  
\end{aufzi}
\end{definition}

Since the mild solution is given by 
\[ 
x(t) = U(t,s)x_s + \int_{s}^{t}U(t,r)B(r)u(r) \dr 
\]
it is clear that  in order to obtain  exact averaged controllability it suffices to consider the case where $x(s) = 0$.

\begin{proposition}\label{prop:equivalences-obs-control}
  Let $B(t) \in \BOUNDED(U,H_{-1})$ be a family of averaged admissible
 controls  for $(A(t))_{t \in [0,\tau]}$. 
  Then
\begin{aufzi}
\item\label{item:thm-equivalences-1} Exact averaged controllability for
  \eqref{eq:controlled-non-autonomous-evolution-equation} in time
  $\tau$ is equivalent to exact averaged observability of the
  retrograde final-value problem (\ref{eq:retrograde}) with the observation operators $C(t) = B(t)^*$.
\item\label{item:thm-equivalences-2} Approximate averaged controllability for
  \eqref{eq:controlled-non-autonomous-evolution-equation} in time
  $\tau$ is equivalent to  approximate averaged observability of the
  retrograde final-value problem (\ref{eq:retrograde}) with the observation operators $C(t) = B(t)^*$.
\item\label{item:thm-equivalences-3} Averaged null controllability for
  \eqref{eq:controlled-non-autonomous-evolution-equation} in time
  $\tau$ is equivalent to averaged observability of $z(s)$,
  $0\le s< \tau$ where $z$ is the solution of the retrograde
  final-value problem (\ref{eq:retrograde}) with the observation operators $C(t) = B(t)^*$.
\end{aufzi}
\end{proposition}
\begin{proof}
  First note that $(\Phi_{s, \tau}^* z_s)(t) = B(t)^* U^*(\tau,t)z_s$
  for $t \in [s,\tau]$. For simplicity we extend this function by zero
  for other values of $t$.
  Exact averaged controllability for
  \eqref{eq:controlled-non-autonomous-evolution-equation} at $\tau$ is
  equivalent to $ \Ran(\Phi_{s,\tau}) = H$ for all $s$. Since these
  operators are bounded, the latter property is equivalent to the fact
  that their adjoints $\Phi_{s,\tau}^*$ is bounded from below on $L^2(s,\tau;H)$, i.e.,
  there exists $\kappa_{s, \tau}$ such that
\[
 \int_{s}^{\tau} \norm{B(t)^*U(\tau,t)^* z_s}^2 \dt \geq \kappa_{s, \tau}\norm{z_s}^2
\]
for all $z_s \in \DOMAIN^*$.  Approximate averaged controllability is
equivalent to $\Ran(\Phi_{s,\tau}) $ being dense for all
$s \in [0, \tau)$, or, equivalently, the respective adjoints being
injective.  Finally, averaged null controllability in time $\tau$ is
equivalent to $\Ran( U(\tau, s) ) \subset \Ran(\Phi_{s,\tau})$ for all
$0 \le s < \tau$.  Applying \cite[Proposition 12.1.2]{Tucsnak-Weiss:book},
averaged null controllability is equivalent to
\[
\norm{  U(\tau, s)^* z_\tau }^2  \le   \delta^2 \norm{ \Phi_{s, \tau}^* z_\tau}^2 = \delta^2 \int_s^\tau 
\bignorm{B(t)^*U(\tau,t)^* z_\tau }^2\dt
\]
for some constant $\delta>0$. But $U(\tau, s)^* z_\tau = z(s)$ where
$z(\cdot)$ is the solution of the retrograde equation
(\ref{eq:retrograde}).
\end{proof}

\section{The averaged Hautus test: skew-adjoint operators } \label{sec:Hautus-1}

Throughout this section, the family of operators
$A(t)_{0\le t \le \tau}$ is as before. Let $C(t)_{0\le t \le \tau}$ be
a family of bounded operators from $\DOMAIN$ to a Hilbert space
$Y$.   In the autonomous case $A(t) = A$ and $C(t) = C$ for all
$t$, it is well known that for admissible $C$ the exact observability
of the system $(A,C)$ implies the so-called Hautus test (or spectral
condition)
\begin{equation} \label{eq:Hautus-2-autonome}
\norm{x}^2 \leq m^2  \norm{C x}^2  + M^2 \norm{(i\xi+A)x}^2
\end{equation}
for some positive constants $m$ and $M$ and all $\xi \in \mathbb{R}$
and $x \in \DOMAIN(A)$. There is also another condition with
$\lambda \in \mathbb{C}$ in place of $i \xi$, see below.  In the
general non-autonomous situation we introduce an integrated (or
averaged) version of this test. We also study, as in the autonomous
case, when the averaged Hautus test is necessary and/or sufficient for
averaged observability.  We start with the "necessary" part.

\begin{proposition}\label{Hautus-necessary}
  Suppose that $(C(t))$ is averaged admissible for $(A(t))$. If the system \eqref{eq:observed-non-autonomous-evolution-equation} is 
  exactly averaged  observable at time $\tau >0$ then there exist
  positive constants $m$ and $M$ such that:
\begin{equation}\label{eq:Hautus-1}\tag{AH.1}
   \norm{x}^2 \leq m^2 \Bigl( \tfrac{1}{\tau} \int_{0}^{\tau}\bignorm{C(s)e^{\lambda s}x}^2 \ds\Bigr) + M^2\Bigl( \tfrac{1}{\tau}\int_{0}^{\tau}e^{\Re \lambda . s}\norm{(\lambda{+}A(s))x}\ds\Bigr)^2
\end{equation}
for all $\lambda \in \mathbb{C}$ and all $x \in \DOMAIN$, 
\begin{equation} \label{eq:Hautus-2}\tag{AH.2}
\norm{x}^2 \leq m^2  \Bigl( \tfrac{1}{\tau} \int_{0}^{\tau}\bignorm{C(s)x}^2\ds\Bigr) + M^2 \Bigl(\tfrac{1}{\tau}\int_{0}^{\tau}\bignorm{(i\xi + A(s))x}^2 \ds\Bigr)
\end{equation}
for all $\xi \in \mathbb{R}$ and $x \in \DOMAIN$.
\end{proposition}
\begin{remark}
  If $C(s) = C$ for all $s$ then \eqref{eq:Hautus-1} can be written as:
  \begin{equation}\label{eq:Hautus-3} \tag{AH.3}
    \norm{x}^2 \leq \tfrac{e^{2\tau\Re(\lambda)} -1}{2\tau \Re(\lambda)}m^2\norm{Cx}^2 +M^2\bigl(\tfrac{1}\tau \int_{0}^{\tau}e^{\Re \lambda . s}\bignorm{(\lambda{+}A(s))x}\ds\bigr)^2
  \end{equation}
If, in addition, $A(s){=}A$ then both assertions coincide with the
classical Hautus (or spectral) conditions. We call the conditions \eqref{eq:Hautus-1} and \eqref{eq:Hautus-2} {\it averaged Hautus tests}. 
\end{remark}
\begin{proof}
  The proof is similar to the autonomous case. We start from 
  $\frac{d}{\ds}\bigl(e^{\lambda s}C(t) U(t,s)x\bigr) = \lambda
  e^{\lambda s}C(t)U(t,s)x + e^{\lambda s}C(t)U(t,s)A(s)x$
  for $x \in \DOMAIN$. Integrating on $[0,\tau]$ yields
\[
    e^{\lambda t} C(t)x - C(t)U(t,0)x = \int_{0}^{t}C(t)U(t,s)(A(s)+ \lambda)x e^{\lambda s}\ds.
\]
Hence,
\[
\int_{0}^{\tau}\bignorm{C(t)U(t,0)x}^2 \dt 
\leq 
2\int_{0}^{\tau}\bignorm{C(t)xe^{\lambda t}}^2 \dt + 2\int_{0}^{\tau}\Bignorm{\int_{0}^{t}C(t)U(t,s)(\lambda{+}A(s))xe^{\lambda s} \ds}^2\dt
\]
Since \eqref{eq:observed-non-autonomous-evolution-equation} is exactly averaged
observable on $[0,\tau]$, the left hand side is bounded below
by $m_0\norm{x}^2$ for some constant $m_0 >0$. We estimate the second
term on the right hand side
\begin{align*}
I 
:= & \; \left( \int_{0}^{\tau}\Bignorm{ \int_0^t C(t)U(t,s)(\lambda{+}A(s))xe^{\lambda s}\ds}^2\dt \right)^{\onehalf}\\
= & \; \sup \left\{ \Bigl| \int_{0}^{\tau}\int_{0}^{t}\Bigl\langle C(t)U(t,s)(\lambda{+}A(s))xe^{\lambda s} , g(t) \Bigr\rangle_{H}\ds\dt \Bigr| \; : \quad \norm{ g }_{L_2(0,\tau;H)} \le 1 \right\}\\
= & \; \sup_{\norm{ g }_{L_2} \le 1}\left|\int_0^\tau  \Bigl\langle (\lambda{+}A(s))x \, e^{\lambda s},  \int_s^\tau  U(t,s)^* C(t)^* g(t) \dt\Bigr\rangle_H \ds\right|\\
\le & \;  \sup_{\norm{ g }_{L_2} \le 1} \Bigl(\int_0^\tau \bignorm{  (\lambda{+}A(s))x \, e^{\lambda s}  }_{H} \; \Bignorm{\int_{s}^{\tau}U(t,s)^* C(t)^* g(t) \dt}_H \ds\Bigr).
  \end{align*}
By Lemma~\ref{lem:admiss-equivalence} and the admissibility assumption of $(C(t))$, there exists a constant $K_\tau > 0$ such that 
\[ 
I  \leq  \;  K_{\tau} \int_{0}^{\tau}\bignorm{(\lambda{+}A(s))xe^{\lambda s}}\ds = \;  K_{\tau} \int_{0}^{\tau}\bignorm{(\lambda{+}A(s))x}e^{\Re \lambda . s}\ds.
\]
and \eqref{eq:Hautus-1} follows. The second assertion is obtained from
the first one by taking $\lambda = i\xi$ and using the Cauchy-Schwarz inequality. 
\end{proof}

Now we study the converse. In the autonomous case i.e.,  $A(s) = A$ and
$C(t) = C$, it is well known that condition \eqref{eq:Hautus-2}
implies the exact observability if the single operator  $A$ is
skew-adjoint. We extend this result to our more general situation.

\begin{theorem}\label{thm:Hautus-skew-adjoint}
  Suppose that $A(t) \in \BOUNDED(\DOMAIN; H)$ be a family
  of skew-adjoint operators generating an evolution family
  $U(t, s)_{0\le s \le t \le \tau}$.  Suppose that the {\em differences} of the
  operators $A(t)$ are bounded and satisfy the  estimate
\[
   \bignorm{A(t)-A(s)}_{\BOUNDED(H)} \leq L \quad \quad \forall t,s \in [0,\tau]
\]
for some constant $L < \frac{1}{\sqrt{2} M}$.  Assume that
$C(t) \in \BOUNDED(\DOMAIN; Y)$ is a family of averaged admissible
observation operators and that the second averaged Hautus condition
\eqref{eq:Hautus-2} holds with positive constants $m$ and $M$.  Then, for all
$\tau > \tau^* := \tfrac{2 \pi M}{ \sqrt{1-2L^2 M^2}}$ there exists
$\kappa_\tau>0$ depending on $M, L$ and $\tau$ such that, for all
$x\in \DOMAIN$ the exact averaged observability estimate
\begin{equation}\label{eq:averaged-observability}
\tfrac1\tau \int_{0}^{\tau}\int_{0}^{\tau}\bignorm{C(s)U(t,0)x}^2 \dt \ds \geq \tfrac{\kappa_\tau}{m^2}
\norm{x}^2 
\end{equation}
holds. In particular, if $C(s) = C$ is constant, then the  system
\eqref{eq:observed-non-autonomous-evolution-equation} is exactly
averaged observable for $\tau > \tau^*$, i.e, for all $x \in \DOMAIN$, 
\[
\int_{0}^{\tau}\bignorm{CU(t,0)x}^2 \dt \;\geq \; \tfrac{\kappa_\tau }{m^2}\; \norm{x}^2.
\]
\end{theorem}
\begin{proof}
  We proceed in a similar way as in the autonomous case. Let $\tau >0$, $\varphi \in H_0^1(0,\tau)$ and $x \in \DOMAIN$. For
  $t,s \in [0,\tau]$, let $h(t):=\varphi(t)U(t,0)x$ and
  $f(t,s):= h'(t)+A(s)h(t)$. Note that $h$ and $f(.,s)$ can be extended
  continuously by zero outside $(0,\tau)$ since
  $\varphi \in H_0^1(0,\tau)$. We write $\widehat{f}(\xi,s)$ for the
  partial Fourier transform of $f$ with respect to the first variable,
  and observe that
\[
\widehat{f}(\xi,s)
=  \;  \int_{\mathbb{R}}e^{-it\xi}f(t,s)\dt 
=  \; \int_{\mathbb{R}}e^{-it\xi}h'(t)\dt + \int_{\mathbb{R}}e^{-it\xi}A(s)h(t)\dt 
\; =  \;  i\xi\widehat{h}(\xi) + A(s)\widehat{h}(\xi)
\]
where we use the fact that each operator $A(s)$ is closed in order to have $\widehat{A(s)h}(\xi) = A(s) \widehat{h}(\xi)$. 
We apply 
\eqref{eq:Hautus-2} with $z_0 = \widehat{h}(\xi)$ to obtain
\begin{align*}
\norm{\widehat{h}(\xi)}^2 
\leq & \; \tfrac{m^2}{\tau}\int_{0}^{\tau}\bignorm{C(s)\widehat{h}(\xi)}^2\ds + \tfrac{M^2}{\tau}\int_{0}^{\tau}\bignorm{(i\xi+A(s))\widehat{h}(\xi)}^2 \ds  \\
= & \; \tfrac{m^2}{\tau}\int_{0}^{\tau}\bignorm{C(s)\widehat{h}(\xi)}^2\ds + \tfrac{M^2}{\tau}\int_{0}^{\tau}\bignorm{\widehat{f}(\xi,s)}^2\ds.
\end{align*}
We integrate over all $\xi \in \RR$ and use Plancherel's theorem
together with the fact that $C(s)\widehat{h}(\xi) = \widehat{C(s)h(\xi)}$ to deduce
\begin{equation}\label{eq:Plancherel-and-Fourier}
\int_{0}^{\tau}\bignorm{h(t)}^2 \dt \leq \tfrac{m^2}{\tau}\int_{0}^{\tau}\int_{0}^{\tau}\bignorm{C(s)h(t)}^2 \dt\ds + \tfrac{M^2}{\tau}\int_{0}^{\tau}\int_{0}^{\tau}\bignorm{f(t,s)}^2 \dt\ds.
\end{equation}
We estimate the last term on the  right hand side as follows
\begin{align}
  & \int_{0}^{\tau}\int_{0}^{\tau}\bignorm{f(t,s)}^2 \dt\ds \notag \\
= & \; \int_{0}^{\tau}\int_{0}^{\tau}\bignorm{h'(t)+A(s)h(t)}^2 \dt\ds \notag \\ 
= & \; \int_{0}^{\tau}\int_{0}^{\tau}\bignorm{\varphi'(t)U(t,0)x-\varphi(t)A(t)U(t,0)x+\varphi(t)A(s)U(t,0)x}^2 \dt\ds \notag \\
\leq & \;  2 \tau \int_{0}^{\tau}\bignorm{U(t,0)x}^2|\varphi'(t)|^2 \dt + 2\int_{0}^{\tau}\int_{0}^{\tau}\bignorm{(A(t)-A(s))U(t,0)x}^2|\varphi(t)|^2 \dt\ds.  \label{eq:split-the-squares}
\end{align}
By skew-adjointness,
\[
\frac{d}{dt}\bignorm{U(t,s)x}^2 = -2\Re\langle A(t)U(t,s)x , U(t,s)x \rangle = 0
\]
for $x\in \DOMAIN$ and so $U(t, s)$ is unitary  for $0\le s \le t \le \tau$. Therefore
 \eqref{eq:Plancherel-and-Fourier} can be rewritten  as
\[ %
 \begin{split}
 \norm{x}^2\int_{0}^{\tau}\bigl|\varphi(t)\bigr|^2 \dt 
 \leq & \; \tfrac{m^2}{\tau}\int_{0}^{\tau}\int_{0}^{\tau}\bignorm{C(s) U(t, 0)x}^2 \varphi(t)^2 \dt\ds + 2 M^2 \norm{x}^2\int_{0}^{\tau}\bigl|\varphi'(t)\bigr|^2 \dt \\
 + & \;  2L^2 M^2 \norm{x}^2 \int_{0}^{\tau}\bigl|\varphi(t)\bigr|^2 \dt.
\end{split}
\] %
Hence
\[
\kappa(\varphi) \, \norm{x}^2 \leq \; \tfrac{m^2}{\tau} \int_{0}^{\tau}\int_{0}^{\tau}\bignorm{C(s)U(t,0)x}^2\bigl|\varphi(t)\bigr|^2 \dt\ds
\]
where 
\[
\kappa(\varphi) = \Bigl((1-2L^2 M^2)\int_{0}^{\tau}\bigl|\varphi(t)\bigr|^2 \dt-2 M^2 \int_{0}^{\tau}\bigl|\varphi'(t)\bigr|^2 \dt \Bigr).
\]
We have to chose $\varphi$ such that  the  constant $\kappa(\varphi)$ is positive. Taking the first eigenfunction of the
Dirichlet Laplacian on $(0,\tau)$, i.e., 
$\varphi(t):=\sin\bigl(\frac{t\pi}{\tau}\bigr)$, we maximize
$\kappa(\varphi)$ and obtain from $\norm{\varphi}_\infty = 1$
 \[
\tfrac{\kappa \tau }{m^2}\; \norm{x}^2 \leq \; \int_{0}^{\tau}\int_{0}^{\tau}\bignorm{C(s)U(t,0)x}^2 \dt\ds 
 \]
where $\kappa = \bigl((1-2L^2 M^2)\frac{\tau}{2}-\tfrac{\pi^2 M^2}{\tau} \bigr)$.
To ensure $\kappa >0$ we need $L^2 < \tfrac{1}{2 M^2}$ and $\tau >  \tau^*$.
\end{proof}

\begin{remark}\label{rem-1}
\begin{enumerate}
\item In \eqref{eq:split-the-squares} we  have used for simplicity  the inequality $(a+b)^2 \le 2 (a^2 + b^2)$ but we could instead use 
$(a+b)^2 \le (1+r) a^2 + (1+r^{-1}) b^2$ for any $r > 0$. In this case, we obtain the theorem (with the same proof) 
with the conditions   
$L < \frac{1}{M \sqrt{1+r}}$ and $\tau^* = \frac{\pi M \sqrt{1 + r^{-1}}}{\sqrt{1 - M^2 (1+r) L^2}}$. 

\item If $A(t)=A$ and hence $L=0$ we obtain (from the previous remark)  as minimal control time
  $\tau^* = \pi M$. This is the usual minimal time  in the case of   unitary groups.  
  
\item In the last assertion of theorem, if instead of $C(s) = C$, we assume that
\[
\norm{C(s)-C(t)} \leq L_0|t-s|^{\alpha}
\]
for some positive constants $\alpha$ and $L_0$ we obtain that for
$L_0$ small enough, the system
\eqref{eq:observed-non-autonomous-evolution-equation} is exactly averaged
observable. Indeed, we have from \eqref{eq:averaged-observability}
\begin{align*}
\kappa \norm{x}^2 
\leq  & \; 2\int_0^\tau\int_0^\tau\bignorm{(C(t)-C(s))U(t,0)x}^2 \ds\dt + 2\int_0^\tau\int_0^\tau\bignorm{C(t)U(t,0)x}^2 \ds\dt \\
\leq  & \; 2L_0  \int_0^\tau\int_0^\tau |t-s|^{2\alpha} \ds\dt \norm{x}^2 + 2\tau\int_0^\tau\bignorm{C(t)U(t,0)x}^2 \dt \\
= & \; \frac{2L_0\tau^{2\alpha+2}}{(2\alpha+1)(\alpha+1)}\norm{x}^2 + 2\tau\int_0^\tau\bignorm{C(t)U(t,0)x}^2 \dt.
\end{align*}
\item If we define 
\[
\widetilde{C}x := \tfrac{1}{\tau}\int_{0}^{\tau}C(s)x \ds
\]
then we can apply 
Proposition~\ref{Hautus-necessary}  and Theorem~\ref{thm:Hautus-skew-adjoint} to the time independent operator
$\widetilde{C}$. We obtain equivalence between 
\[
\kappa_{\tau}\bignorm{x}^2 \leq \int_{0}^{\tau}\Bignorm{\int_0^\tau C(s)U(t,0)x \ds}^2 \dt
\]
and 
\[
\norm{x}^2 \leq m^2\Bignorm{\Bigl( \tfrac{1}{\tau} \int_0^\tau C(s)\ds \Bigr)x }^2 + M^2 \Bigl(\tfrac{1}{\tau} \int_{0}^{\tau}\bignorm{(i\xi-A(s))x}^2 \ds \Bigr).
\]
\item We have assumed in the theorem that $A(t)$ are skew-adjoint operators in order to have $U(t,s)$ is a unitary operator on $H$. The previous proof works under the assumption that 
\begin{equation*}
K_0 \norm{x} \le \norm{U(t,0)x} \le K_1 \norm{x}, \, x \in H
\end{equation*}
for some positive constants $K_0$ and $K_1$. The statement of the theorem holds with different conditions $L$ and $\tau^*$ (depending on $K_0$ and $K_1$).
\end{enumerate}
\end{remark}

\section{The averaged Hautus test: a more general class of operators }\label{sec:Hautus-2}

In this section we extend  Theorem \ref{thm:Hautus-skew-adjoint} to a more general class of operators. More precisely, we consider operators $A(t)$ for which the corresponding evolution family $U(t,s)$ is not necessarily an  isometry but satisfies an estimate of the form
\begin{equation}\label{JZ}
k e^{\alpha(t-s)} \norm{x} \le \norm{U(t,s)x} \le K e^{\beta(t-s)} \norm{x}, \, x \in H
\end{equation}
for some constants $k, K, \alpha$ and $\beta$.  This question was  considered in the autonomous case $A(t) = A$ and $C(t) = C$ by  Jacob and Zwart \cite{Jacob3}. We shall follow similar ideas as in their paper. Note however, even in this autonomous case, the result is very much less precise than in the case of unitary groups. In particular, the minimal time for observability obtained  in \cite{Jacob3} is  $\frac{1}{\beta- \alpha}$. This  value becomes large as $\alpha$ and $\beta $ are close and this is not consistent with the result on unitary groups. 

The main tool is the following optimal Hardy inequality.

\begin{theorem}[ {Gurka \cite{Gurka}, Opic-Kufner \cite{OpicKufner}} ]
  \label{thm:gurka-opic-kufner}
  Let $v, w \ge 0$ be weight functions on $[0, \tau]$.  Then the
  weighted Hardy inequality
  \begin{equation}    \label{eq:hardy-optimal}
     \bignorm{ \varphi }_{L_2(0, \tau; w(x)dx)} \le C_H \bignorm{ \varphi' }_{L_2(0, \tau; v(x)dx)}
  \end{equation}
  holds for all $\varphi \in H_0^1(0, \tau)$ if and only if
\[
 B := \sup\left\{ \left(\int_x^y w(t)\dt\right) \;  \min\left( \int_0^x \tfrac{1}{v(t)} \dt , \int_{y}^\tau \tfrac{1}{v(t)} \dt \right) :\quad 0 < x, y < \tau \right\}
\]
is finite. In this case, the optimal constant $C_H$ in
\eqref{eq:hardy-optimal} satisfies
$\tfrac{B}{\sqrt{2}} \le C_H \le 4 B$.
\end{theorem}

We make a basic remark on evolution families
$U(t, s)_{0\le s\le t}$. Given $U(t,s)$ which is exponentially  bounded, i.e., 
$\norm{ U(t, s) x} \le K e^{\beta (t-s)} \norm{x}$. If in addition  
each  $U(t, s)$ is  invertible then  writing
$V(t) := U(t, 0)$ gives
\[
V(t) = U(t, 0) = U(t, s) U(s, 0) = U(t, s) V(s)
\quad\Leftarrow\joinrel=\joinrel=\joinrel\Rightarrow\quad
U(t, s) = V(t) V(s)^{-1}.
\]
Then $I = V(t) V(t)^{-1}$ gives $\norm{x} \le K e^{\beta t} \norm{V(t)^{-1} x}$ and so
$\norm{ V(t)^{-1} x } \ge \tfrac1K e^{-\beta t} \norm{x}$ so that
  \begin{equation}    \label{eq:group-estimate}
  k e^{\alpha (t{-}s)} \norm{x} \le \norm{ U(t, s) x } \le K e^{\beta (t{-}s)} \norm{x}.
  \end{equation}
  holds for $\alpha={-}\beta$ and $k = \tfrac1K$. If $A$ is 'shifted',
  i.e.,  replaced by $A{+}\omega$, this symmetry $\alpha={-}\beta$ will
  break, and we will therefore use only \eqref{eq:group-estimate} for
  {\em some} constants $k, K >0$ and $\alpha \le \beta$.

\begin{theorem}\label{thm:Hautus-invertible-family}
  Let $A(t)_{0\le t \le \tau} \in \BOUNDED(\DOMAIN; H)$ be a family of
  operators generating an evolution family $U(t, s)$ and
  let $0 < k \le K$ and $\alpha < \beta$ be such that
  \eqref{eq:group-estimate} holds. We
  suppose that the  differences $A(t) - A(s)$ are bounded operators with 
  $\norm{ A(t) - A(s) }\le L$ for  some $L$ such that $L < \displaystyle \frac{k}{\sqrt{2} K M e^{(\beta-\alpha)  \tau}}$. 
Let $C \in \BOUNDED(\DOMAIN; Y)$.  Then the averaged Hautus condition
  \eqref{eq:Hautus-3} implies exact observability for all $\tau > \tau^{**}$, i.e., 
\[
\int_{0}^{\tau}\bignorm{CU(t,0)x}^2 \dt \geq \tfrac{\kappa}{m^2 } \; \bignorm{x}^2 \quad \forall x \in H
\]
for some $\tau^{**} > 0$ provided that there exist $0 \le x \le y \le  \tau^{**}$ such that
\[
 f(x, y) := \; \left( \tfrac{k^2}{4K^2 M^2 (\beta- \alpha)} (e^{-2(\beta- \alpha) x} - e^{-2(\beta- \alpha) y})  + L^2 (x-y)   \right) \min(x, \tau-y) > 2.
\]
\end{theorem}

\begin{proof}
  Observe that exact (averaged) observability  is
  invariant under spectral shifts (replacing $A$ by $A{+}\omega$),
  which in turn allows to assume $\beta=0$ and $\alpha={-}\omega$ for
  $\omega = \beta-\alpha > 0$.  
 We  follow the lines of the proof of
  Theorem~\ref{thm:Hautus-skew-adjoint} until
  \eqref{eq:split-the-squares}.  Using \eqref{eq:group-estimate}
  instead of unitarity leads to consider a new function
\[
\kappa(\varphi) := \int_0^\tau |\varphi(t)|^2 ( k^2 e^{-2\omega t} - 2
K^2 M^2 L^2 ) \,dt \; - \; 2 K^2 M^2 \int_0^\tau
|\varphi'(t)|^2 \,dt .
\]
Then $\kappa(\varphi) > 0 $ is equivalent to
\begin{equation}  \label{eq:inverse-hardy}
 \int_0^\tau |\varphi'(t)|^2 \,dt
<  \int_0^\tau |\varphi(t)|^2 ( \tfrac{k^2}{2 K^2 M^2} e^{-2\omega t} - L^2 )\,dt.
\end{equation}
This is an 'inverse Hardy inequality', when compared to
\eqref{eq:hardy-optimal}. To establish such an estimate for at least
one function $\varphi \in H_0^1(0, \tau)$, we consider on $[0, \tau]$
the weight function
\[
  w(t) = \tfrac{k^2}{2 K^2 M^2} e^{-2\omega t} - L^2.
\]
Observe that $w$ is positive if 
\begin{equation}  \label{eq:Lipschitz}
0 \le L < \displaystyle \frac{k}{\sqrt{2} K M e^{(\beta-\alpha)  \tau}}.
\end{equation}
In order to obtain \eqref{eq:inverse-hardy} we use the optimality statement
in Theorem~\ref{thm:gurka-opic-kufner} with $v(x) = 1$: if
$\sqrt{2}< B < \infty$, the optimal constant guaranteeing
\eqref{eq:hardy-optimal} is larger than one.  Hence, for any $C<1$
there exists a  $\varphi \in H_0^1([0, \tau])$ for which 
\eqref{eq:hardy-optimal} fails. This function will then satisfy
\eqref{eq:inverse-hardy}, and provides a strictly positive constant
$\kappa(\varphi)$, yielding exact averaged observability with
$\kappa := \kappa(\varphi)$, as in the proof of
Theorem~\ref{thm:Hautus-skew-adjoint} (by rescaling we may suppose
$\norm{\varphi}_\infty = 1$). Clearly, $\sqrt{2}< B$ is equivalent to
our condition on $f(x, y)$ to be larger than 2 for some $0 \le x \le y$.
\end{proof}

On the compact set $T = \{ 0 \le x \le y \le \tau \} \subset \RR^2$ we
consider the function
\begin{align*}
f(x, y) & := \; \left(\int_x^y w(t)\,dt\right) \;  \min\left( \int_0^x\dt , \int_{y}^\tau  \,dt \right)\\
& = \; \left( \tfrac{k^2}{4K^2 M^2 \omega} (e^{-2\omega x} - e^{-2\omega y})  + L^2 (x-y)   \right) \min(x, \tau-y).
\end{align*}
It is continuous and satisfies $f |_{\partial T} = 0$ so that the maximum is
taken inside $T$. However, due to the many parameters and the mixture
of power-type functions with exponentials it may be difficult to
calculate explicitly the maximum of $f$ in $T$. We therefore
concentrate on a sufficient condition that ensures
$f(x, y) > 2$ for some $x$ and $y$. We consider for example the case where 
  $f(\tfrac14 \tau , \tfrac34 \tau) > 2$, i.e., 
\[
   \tau  e^{-\frac{\omega\tau }2} \left(\tfrac{1- e^{-\omega \tau}}{\omega \tau} \right)  
>  \tfrac1\tau \Bigl(\tfrac{32 K^2 M^2}{k^2}\Bigr) + \tau \Bigl(\tfrac{2L^2 K^2 M^2}{k^2}\Bigr).
\]
By numerical calculations\footnote{The function $g(x) = e^{-\nicefrac{x}{2}}(\tfrac{1- e^{-x}}{x})$ is larger than $\onehalf$ for $x \le 0.7143$ and $\tfrac{1}{\sqrt{2}}\le 0.70711$.}, we see that if
$ \omega \tau \le \tfrac{1}{\sqrt{2}}$, then the left hand side is larger than
$\tfrac{\tau}2$, so that for $\tfrac{2L^2 K^2 M^2}{k^2} < \tfrac14$,
$\tau^2 = \frac{128 K^2 M^2}{k^2}$ gives a concrete observation time. We obtain the  following corollary.  

\begin{corollary}\label{cor-hautus}
 Suppose that $L < \tfrac{k}{2 \sqrt{2} KM}$ and $0 \le  \beta - \alpha \le \tfrac{k}{16KM}$. Then we have exact observability  at  time $\tau > \tau^{**}$ 
where $\tau^{**} = \tfrac{8 \sqrt{2}KM}{k}$. In particular, if $k{=}K{=}1$ and $L, M$ are such that
$8 L^2 M^2 < 1$ and  $0 \le \beta - \alpha  \le \tfrac{1}{16 M}$,
then we have exact observability  at  time $\tau > \tau^{**}$ 
where $\tau^{**} = 8 \sqrt{2}M$.
\end{corollary}

In the autonomous case $A(t) = A$ with $A$ is a generator of a  group we have  $L=0$, hence  for
$0 \le \beta - \alpha  \le \tfrac{k}{16 K M}$ we obtain  exact observability  at time
$\tau > \tau^{**} = \tfrac{8 \sqrt{2}KM}{k}$. This might be better than the observation time given in \cite{Jacob3}  which is 
$\frac{1}{\beta-\alpha}$.

\section{Applications to the wave and Schrödinger equations with time dependent potentials}\label{sec:applications}

In this section we give applications of our results to observability
of the Schrödinger and wave equations both with time dependent
potentials. We also consider the damped wave equation with time
dependent damped term. Before going into these examples we explain the
general idea. It is based on a perturbation argument which shows that
the Hautus test carries over from the time independent operator to
time dependent ones. Once the Hautus test is satisfied by the
perturbed operator we appeal to the results of the previous sections
and obtain observability of the system.
 
Let $A$ be the generator of unitary group on  $H$. We assume that
$C : \DOMAIN(A) \to Y$ is an admissible operator and such that the
system $(A,C)$ is exactly observable at time $\tau_0$. Therefore the
Hautus test is satisfied by the operators $A$ and $C$. Now let
$R(t)_{0\le t \le \tau}$ be a family of uniformly bounded operators on
$H$. By classical bounded perturbation argument (see, e.g., \cite[Theorem 9.19]{EngelNagel}). 
the operators given by $A(t) = A{+}R(t)$, $t \in [0, \tau]$, 
generate an evolution family $U(t,s)$ on $H$. Note that for every
$x \in H$
\begin{equation}\label{eq:quasi-contr}
e^{-\beta(t-s)} \norm{x} \le \norm{U(t,s) x} \le e^{\beta(t-s)} \norm{x}
\end{equation}
with $\beta = \sup_{t \in [0, \tau]} \norm{R(t)}$. Indeed, one has for every $x \in \DOMAIN(A)$, 
$\Re \langle (A + R(t))x, x \rangle = \Re \langle R(t) x, x \rangle$ ans hence 
 \[
- \beta \norm{x}^2 \le  \Re \langle (A + R(t))x, x \rangle \le \beta \norm{x}^2.
\]
We apply this with $U(t,s)x$ at the place of $x$ and obtain 
\[
- \beta \norm{U(t,s) x}^2 \le \tfrac{1}{2} \tfrac{\partial}{\partial t} \bignorm{U(t,s)x}^2 \le \beta \norm{U(t,s) x}^2.
\]
We integrate and obtain \eqref{eq:quasi-contr}. Note that if $\Re \langle R(t) x, x \rangle = 0$, then $U(t,s) $ is unitary. 

Let now $x \in \DOMAIN(A)$ and $\xi \in \mathbb{R}$. The Hautus test for $(A,C)$ gives
\begin{eqnarray*}
\norm{x}^2 &\le& m^2 \norm{Cx}^2 + M^2 \norm{ (i\xi + A)x}^2\\
&\le&  m^2 \norm{Cx}^2 + 2M^2 \norm{ (i\xi + A + R(s))x}^2 + 2M^2 \norm{ R(s)}^2 \norm{x}^2.
\end{eqnarray*}
Integrating on $[0, \tau]$ with respect to $s$ gives
\[
\norm{x}^2 \le m^2 \norm{Cx}^2 + 2M^2 \Bigl(\tfrac{1}{\tau} \int_0^\tau \bignorm{ (i\xi + A + R(s))x}^2 \ds \Bigr) + 2M^2 \Bigl( \tfrac{1}{\tau}\int_0^\tau \bignorm{ R(s)}^2 \ds \Bigr)  \norm{x}^2.
\]
Suppose  in addition that there exists $\tau_1 > 0$ and $\mu < 1$ such that for $\tau \ge \tau_1$
\begin{equation}\label{eq:hyp-small}
2M^2 \Bigl( \tfrac{1}{\tau}\int_0^\tau \bignorm{ R(s)}^2 \ds \Bigr) \le \mu.
\end{equation}
Then we obtain 
\begin{equation}\label{hautusA-R}
(1- \mu) \norm{x}^2 \le m^2 \norm{Cx}^2 + 2M^2 \Bigl(\tfrac{1}{\tau} \int_0^\tau \bignorm{ (i\xi + A + R(s))x}^2 \ds \Bigr).
\end{equation}
Note that we could also replace $i \xi$ by $\lambda \in \mathbb{C}$
and obtain the Hautus test \eqref{eq:Hautus-3}.   
Next we assume that $C$ is admissible for the unitary group $e^{tA}$ generated by $A$. That is there exists a constant $K_\tau > 0$ such that
\begin{equation}\label{admissC-A}
\int_0^\tau \bignorm{Ce^{tA}x}^2 \dt \le K_\tau \norm{x}^2, \ x \in \DOMAIN(A).
\end{equation}
We prove that $C$ is admissible for $(A+ R(t))$. In order to do so, we start from Duhamel's formula\footnote{in order to prove this formula one takes the derivative of $f(r) := e^{(t- r)A} U(r,s)x$ for $s \le r \le t$ and then integrate from $s$ to $t$.}
\begin{equation}\label{duhamel}
U(t,s) x - e^{(t-s)A}x = \int_s^t e^{(t-r)A} R(r) U(r,s)x \dr.
\end{equation}
We use \eqref{admissC-A}  so that 
\begin{align*}
\int_0^\tau \bignorm{CU(t,s)x}^2 \dt
\leq  & \; 2 \int_0^\tau \bignorm{C e^{(t-s)A} x}^2 \dt + 2 \int_0^\tau \Bignorm{ \int_s^t C e^{(t-r)A} R(r) U(r,s)x \dr}^2 \dt \\
\leq  & \;  2 K_\tau \norm{x}^2 + 2 \tau \int_s^\tau \int_r^\tau \bignorm{  C e^{(t-r)A} R(r) U(r,s)x }^2 \dt \dr \\
\leq & \; 2 K_\tau \norm{x}^2 + 2 K_\tau \int_s^\tau  \norm{R(r) U(r,s)x }^2 \dr \leq  \;  K'_\tau \norm{x}^2,
\end{align*}
where we use the fact that the operators $R(r)$ are uniformly bounded and $U(t,s)$ is exponentially  bounded.  \\
 We have admissibility of $C$ and the averaged Hautus test \eqref{hautusA-R}. Now we conclude
either by Theorem \ref{thm:Hautus-skew-adjoint} or Corollary
\ref{cor-hautus} that, as soon as $\norm{R(t) -R(s)}$ are small enough,
we have exact observability of the system $(A + R(.), C)$ at time
$\tau > \tau^*$ for some $\tau^* > 0$. Note that \eqref{eq:hyp-small}
holds if $R(t) = 0$ for $t \ge t_0$ for some $t_0 > 0$.

\vspace{.4cm}
\noindent \underline{The Schrödinger equation.} Let $\Omega$ be a
bounded domain of $\mathbb{R}^d$ with a $C^2$-boundary $\Gamma$. Let
$\Gamma_0$ be an open subset  of $\Gamma$ and $Y = L^2(\Gamma_0)$. It is known
that for appropriate condition on $\Gamma_0$, the Schrödinger
equation
\begin{equation}\label{eq:schro}
 \left\{
   \begin{array}{lcl}
   z'(t,x) &=& i \Delta z(t,x) \quad (t,x)  \in [0,\tau] \times \Omega\\
   z(0,.) &=& z_0 \in H^2(\Omega) \cap H^1_0(\Omega)\\ 
   z(t,x) &=& 0 \quad (t,x) \in [0, \tau]\times \Gamma
   \end{array}\right.
 \end{equation}
 satisfies the observability inequality
 \begin{equation}\label{eq:obs-schro}
 \int_0^\tau \int_{\Gamma_0} \vert \tfrac{\partial z}{\partial \nu}(t,x) \vert^2 d\sigma \dt \ge \kappa_\tau \norm{z_0}_{H^1_0(\Omega)}^2
 \end{equation}
 for every $\tau > 0$, see for example \cite[Chapter 7]{Tucsnak-Weiss:book}.   Let $C$ be the normal derivative
 $\tfrac{\partial }{\partial \nu}$ on $\Gamma_0$,
 $Y = L^2(\Gamma_0, d\sigma)$ and $\Delta_D$ the Laplacian with
 Dirichlet boundary conditions.  The previous inequality means that
 the system $(i\Delta_D, C)$ is exactly observable at time $\tau$. Let
 now $R(t) f = i V(t)f$ where $V(t, .) \in W^{1,\infty}(\Omega)$ is
 a real-valued potential which depends on time. Then under appropriate
 conditions on $V$ we obtain from the discussion above that the
 non-autonomous system $(i(\Delta_D + V(t)), C)$ is exactly observable
 at time $\tau > \tau^*$ for some $\tau^* > 0$. This means that
 \eqref{eq:obs-schro} is satisfied for the solution of the
 Schr\"odinger equation with time dependent potential
 \begin{equation}\label{eq:schro-t-}
 \left\{
   \begin{array}{lcl}
   z'(t,x) &=& i \Delta z(t,x) + i V(t) z(t,x) \quad (t,x)  \in [0,\tau] \times \Omega\\
   z(0,.) &=& z_0 \in H^2(\Omega) \cap H^1_0(\Omega)\\ 
   z(t,x) &=& 0 \quad (t,x) \in [0, \tau]\times \Gamma. 
   \end{array}\right.
 \end{equation}
 Note however that our method does not  give observability at any time
 $\tau > 0$.  If $V(t) = V$ is independent of $t$ then observability
 for the Schr\"odinger equation perturbed by the potential $V$ holds
 at any time $\tau > 0$, see \cite[Chapter 7]{Tucsnak-Weiss:book} and the references there.

\vspace{.4cm}

 \noindent \underline{The wave equation.} Let again $\Omega$ be a
 bounded smooth domain of $\mathbb{R}^d$. We consider the wave
 equation
 \begin{equation}\label{eq:wave}
 \left\{
   \begin{array}{lcl}
   z''(t,x) &=&  \Delta z(t,x)   \in [0,\tau] \times \Omega\\
   z(0,.) &=& z_0 \in H^1_0(\Omega), \, z'(0,.) = z_1 \in L^2(\Omega)\\ 
   z(t,x) &=& 0 \quad (t,x) \in [0, \tau]\times \Gamma. 
   \end{array}\right.
 \end{equation}
 Let $\Gamma_0$ be a part of the boundary $\Gamma$. Observability for
 the wave equation with the observation operator
 $C = \tfrac{\partial}{\partial \nu}_{\vert \Gamma_0}$ have been
 intensively studied. Under appropriate geometric conditions on
 $\Gamma_0$, there exists $\tau_0 > 0$ such that for $\tau > \tau_0$
 there exists a positive constant $\kappa_\tau$ such that
\begin{equation}\label{eq:obs-wave}
\kappa_\tau \Bigl( \int_\Omega | z_1 |^2 + \int_\Omega | \nabla z_0 |^2 \Bigr) \le \int_0^\tau \int_{\Gamma_0} |\tfrac{\partial z}{\partial \nu} |^2 d\sigma\dt.
\end{equation}
We refer to \cite{bardos-l-r,lions,Komornik:livre} and the references therein.
Let $ A_0 = \begin{pmatrix}
0& I \\
-\Delta_D &0
\end{pmatrix}$
on $H := H^1_0(\Omega) \times L^2(\Omega)$. It is a standard fact that
$A_0$ generates a unitary group $U(t)_{t\in \mathbb{R}}$ on $H$.  Set
$\widetilde{C} (f,g) := (\tfrac{\partial f}{\partial \nu}_{\vert
  \Gamma_0}, 0)$.
Then the energy estimate \eqref{eq:obs-wave} is precisely the
observability inequality
\begin{equation}\label{eq:obs-wave2}
\kappa_\tau \norm{(z_0,z_1)}^2_H \le \int_0^\tau \bignorm{\widetilde{C}U(t)(z_0,z_1)}^2_{L^2(\Gamma_0)}\dt.
\end{equation}
Now we consider the damped wave equation without a potential
\begin{equation}\label{eq:wave-dam}
 \left\{
   \begin{array}{lcl}
   z''(t,x) &=&  \Delta z(t,x) + b(t,x)z'(t,x) + V(t,x) z(t,x)   \in [0,\tau] \times \Omega\\
   z(0,.) &=& z_0 \in H^1_0(\Omega), \, z'(0,.) = z_1 \in L^2(\Omega)\\ 
   z(t,x) &=& 0 \quad (t,x) \in [0, \tau]\times \Gamma. 
   \end{array}\right.
 \end{equation}
Going to the first order system on $H$, the wave equation \eqref{eq:wave-dam} can be rewritten as $Z' = A(t) Z$ with  
$ A(t) = \begin{pmatrix}
0& I \\
\Delta + V(t) & b(t)
\end{pmatrix} = A_0 + R(t)$
where  $R(t)= \begin{pmatrix}
0& 0 \\
V(t) & b(t)
\end{pmatrix}$.
As in the case of the Schrödinger equation we can apply the previous
discussion to see that the Hautus test for $A_0$ implies our averaged
Hautus test for $(A(t))_t$.  In order to do so we need to verify
\eqref{eq:hyp-small}. This property holds if
\[
\frac{1}{\tau} \int_0^\tau \Bigl( \norm{V(t)}^2_{W^{1,\infty}(\Omega)} + \norm{b(t)}^2_{L^\infty(\Omega)} \Bigr)\dt
 \]
 is small enough. The norms $\norm{R(t)-R(s)}$ are small if the
 quantities
 $ \norm{V(t)-V(s)}_{W^{1,\infty}(\Omega)} +
 \norm{b(t)-b(s)}_{L^\infty(\Omega)} $
 are small. In this case, we obtain exact averaged observability for
 \eqref{eq:wave-dam}. That is, we obtain the energy estimate
 \eqref{eq:obs-wave} for $\tau$ large enough for solution $z$ to
 \eqref{eq:wave-dam}.  If $V$ and $b$ are independent of $t$ then
 observability results are known (see \cite{Tucsnak-Weiss:book}). If
 $b(t) = 0$ and $V$ depends on $t$, then a more precise result can be
 found in \cite{puel} for a special class of $\Gamma_0$. The proof in
 \cite{puel} is different from ours and it is based on Carleman
 estimates.

\end{document}